\documentclass{amsart}
\usepackage{amsmath,amssymb,amsfonts}
\usepackage{listings}
\usepackage{comment}
\usepackage{courier}
\usepackage{mathrsfs} 
\usepackage{enumerate}
\usepackage{color}
\usepackage{graphicx}
\usepackage[usenames,dvipsnames]{xcolor}
\usepackage{hyperref}
\newcommand{\bburl}[1]{\textcolor{blue}{\url{#1}}}
\usepackage[utf8]{inputenc}
\usepackage[T1]{fontenc}
\hypersetup{colorlinks=false,urlcolor=MidnightBlue,citecolor=PineGreen,linkcolor=BrickRed, linktoc=all}

\numberwithin{equation}{section}
\newcommand{\E}{\mathcal{E}}

\newcommand{\Fp}{\mathbb{F}_p}
\newcommand{\Frob}{\text{Frob}}
\newcommand{\Gal}{\text{Gal}}

\newcommand{\legendre}[2]{\left(\frac{#1}{#2}\right)}
\newcommand{\nequiv}{\not\equiv}

\newcommand{\PP}{\mathbb{P}}
\newcommand{\QQ}{\mathbb{Q}}
\newcommand{\rank}{\operatorname{rank}}

\newcommand{\X}{\mathcal{X}}
\newcommand{\ZZ}{\mathbb{Z}}
\newcommand{\Z}{\mathbb{Z}}

\usepackage{amsthm}
\newtheorem{theorem}{Theorem}[section]
\newtheorem{lemma}[theorem]{Lemma}

\newtheorem{corollary}[theorem]{Corollary}
\newtheorem{remark}[theorem]{Remark}
\newtheorem{conjecture}[theorem]{Conjecture}
\newtheorem{example}[theorem]{Example}

\title{Rank and Bias in Families of Hyperelliptic Curves via Nagao's Conjecture}
\date{\today}
\author{Trajan Hammonds, Seoyoung Kim, Benjamin Logsdon, \'Alvaro Lozano-Robledo, Steven J. Miller}

\begin{document}

\maketitle

\begin{abstract}
Let $\mathcal{X} : y^2 = f(x)$ be a hyperelliptic curve over $\QQ(T)$ of genus $g\geq 1$. Assume that the jacobian of $\mathcal{X}$ over $\mathbb{Q}(T)$ has no subvariety defined over $\mathbb{Q}$. Denote by $\mathcal{X}_t$ the specialization of $\mathcal{X}$ to an integer $T=t$, let $a_{\X_t}(p)$ be its trace of Frobenius, and $A_{\mathcal{X},r}(p) = \frac{1}{p}\sum_{t=1}^p a_{\X_t}(p)^r$ its $r$-th moment. The first moment is related to the rank of the jacobian $J_\mathcal{X}\left(\mathbb{Q}(T)\right)$ by a  generalization of a conjecture of Nagao: 
$$\lim_{X \to \infty} \frac{1}{X} \sum_{p \leq X} - A_{\mathcal{X},1}(p) \log p = \rank  J_\mathcal{X}(\mathbb{Q}(T)).$$
Generalizing a result of S. Arms, \'A. Lozano-Robledo, and S.J. Miller, we compute first moments for various families resulting in infinitely many hyperelliptic curves over $\mathbb{Q}(T)$ having jacobian of moderately large rank $4g+2$, where $g$ is the genus; by Silverman's specialization theorem, this yields hyperelliptic curves over $\mathbb{Q}$ with large rank jacobian. Note that Shioda has the best record in this directon: he constructed hyperelliptic curves of genus $g$ with jacobian of rank $4g+7$. In the case when $\mathcal{X}$ is an elliptic curve, Michel proved $p\cdot A_{\mathcal{X},2} = p^2 + O\left(p^{3/2}\right)$. For the families studied, we observe the same second moment expansion. Furthermore, we observe the largest lower order term that does not average to zero is on average negative, a bias first noted by S.J. Miller in the elliptic curve case. We prove this bias for a number of families of hyperelliptic curves.
\end{abstract}



\section{Introduction}


Given an elliptic curve $E/\QQ: y^2 = x^3 + Ax + B$ with $A$ and $B$ integers, the Mordell--Weil theorem shows that the set of rational solutions $E(\mathbb{Q})$ forms a finitely generated abelian group. Mazur \cite{Maz1,Maz2} proved there are only fifteen possibilities for the torsion subgroup, all of which occur for infinitely many non-isomorphic elliptic curves over $\QQ$.  However, much less is known about the rank. Recent breakthroughs, such as \cite{Bh, BhSh1, BhSh2}, have shown that the average rank among all elliptic curves over $\QQ$ is bounded by $7/6$, and that a positive percentage of curves are rank $0$, but it is still unknown if the rank is unbounded as we vary over all curves. The largest known rank is at least 28, due to Noam Elkies \cite{E}, and some recent models (see \cite{PPVW}) suggest that the rank may in fact be bounded (interestingly, their prediction for the largest rank is very close to the largest observed!).

If the Birch and Swinnerton-Dyer conjecture \cite{BSD1, BSD2} holds, then the order of vanishing of the Hasse--Weil $L$-function $L(E/\QQ,s)$ at the central point (the analytic rank) equals the number of generators of the Mordell--Weil group (the algebraic rank). Conjecturally, however, there are other relations between the traces of Frobenius at each prime $p$ and the algebraic rank. Nagao posited that the first moment sums in a one-parameter family of elliptic curves determine the rank of the elliptic surface. More concretely, let \begin{equation} \mathcal{E}: y^2 \ = \ x^2 + A(T)x + B(T) \end{equation} be an elliptic curve over $\mathbb{Q}(T)$ and, for an integer $t\in\ZZ$, let \begin{equation} a_{\mathcal{E}_t}(p) \ = \ p + 1 - \# \mathcal{E}_t(\Fp),  \end{equation} where $\# \mathcal{E}_t(\Fp)$ is the number of points over $\Fp$ on the specialization of $\mathcal{E}$ at $T=t$. Also, for each $r\geq 1$, we define the $r$\textsuperscript{th} moment of the traces of Frobenius by \begin{equation} A_{\mathcal{E},r}(p) \ = \  \frac1{p} \sum_{t=0}^{p-1} a_{\mathcal{E}_t}(p)^r. \end{equation} Then, Nagao \cite{Na3} conjectured that \begin{equation} \lim_{P \to \infty} \frac{1}{P} \sum_{p \leq P} -A_{\mathcal{E},1}(p) \log p\ =\ {\rm rank} \ \mathcal{E}(\mathbb{Q}(T)). \end{equation} Rosen and Silverman \cite{RS} have proved that Nagao's conjecture holds for surfaces where Tate's conjecture holds, which includes rational surfaces\footnote{The elliptic surface is rational iff one of the following is true: (1) $0 < \max\{3 \mbox{deg} A, 2\mbox{deg} B\} < 12;$ (2) $3\mbox{deg} A = 2\mbox{deg} B = 12$ and ${\rm ord}_{t=0}t^{12} \Delta(t^{-1}) = 0$. See \cite{RS}, pages $46-47$ for more details.}. The second author \cite{kim} has shown that the Sato-Tate conjecture  implies Nagao's conjecture for certain twist families of elliptic curves, and also in a generalized form for hyperelliptic curves, that we shall describe below. 

There has been a long history of attempts to construct either individual elliptic curves with large rank, or families with large rank. Many of these were found by looking for curves where these associated sums are large and then analyzing these curves carefully; however, the work of Nagao, Rosen and Silverman presents another approach. If one can find a family of elliptic curves so that the sum $A_{\mathcal{E},1}(p)$ is computable, then we would conjecturally have computed (unconditionally if Tate's conjecture is true for the surface) the rank of the family. Thus, the challenge is to find choices of $A(T)$ and $B(T)$ so that the resulting moments $A_{\mathcal{E},1}(p)$ can be computed and are negative, and large in absolute value. This was first done by Arms, Lozano-Robledo and Miller \cite{ALM, Mil1, Mil2}, where elliptic curves over $\QQ(T)$ of rank up to 8 were constructed\footnote{Up to rank 6 the constructions gave rational surfaces and the results were unconditional; for the larger rank the surfaces were not rational, but one can isolate candidate points from the method, and then directly show that these are linearly independent. The rank 8 case would correspond in this setting to rank 4g+4, but we do not pursue this in this paper.}, and then generalized in \cite{MMRSY} to function fields over number fields. In this paper, we are interested in constructing jacobians of hyperelliptic curves over $\QQ$ with high rank, conditional on the following generalization of Nagao's conjecture. 

\begin{conjecture}
	\label{conj-nagao-new}
	Let $\X$ be a hyperelliptic curve defined over $\QQ(T)$. Assume that the jacobian of $\mathcal{X}$ over $\mathbb{Q}(T)$ has no subvariety defined over $\mathbb{Q}$. For an integer $t\in\ZZ$, let $\#\X_t(\Fp)$ be the number of points over $\Fp$ on the specialization of $\X$ at $T=t$, and for each prime $p\geq 2$, we define $a_{\X_t}(p)=p+1-\#\X_t(\Fp)$, and $A_{\X,1}(p)= \frac1{p} \sum_{t=0}^{p-1} a_{\X_t}(p).$ Then, we have 
	\begin{equation} \lim_{P \to \infty} \frac{1}{P} \sum_{p \leq P} -A_{\X,1}(p) \log p\ =\ {\rm rank} \ J_\X(\mathbb{Q}(T)). \end{equation} 
\end{conjecture}

We recall here that the Chow $\QQ(T)/\QQ$ trace of an abelian variety $A/\QQ(T)$ is a pair $(\tau,B/\QQ)$, where $B$ is an abelian variety over $\QQ$, and $\tau\colon B\to A$ is a homomorphism of abelian varieties defined over $\QQ(T)$, with the universal mapping property that, given any such pair $(\tau',B')$, the map $\tau' \colon B'\to A$ should factor through $\tau$ (see \cite{C} for more details on Chow traces). In particular, we note that if $A/\QQ(T)$ has no subvariety defined over $\QQ$, then its Chow trace is necessarily trivial.

Hindry and Pacheco \cite{HP} have given a more general version of Nagao's conjecture for a projective surface with a fibration onto a curve. We shall discuss below the relation between their conjecture and Conjecture \ref{conj-nagao-new}.

In our main theorem, and for each fixed genus $g\geq 1$, we construct a hyperelliptic curve $\X: y^2= f_g(x,T)$ such that Nagao's limit is computable, and if we assume Conjecture \ref{conj-nagao-new}, then its jacobian has rank $4g+2$. We note that preliminary data of hyperelliptic curves ordered by discriminant up to size $10^6$ and $10^7$, due to Sutherland \cite{Su}, show that $95.68\%$ of genus $2$ curves (resp. $92.52\%$ of genus $3$ curves) have analytic rank $0$, $1$, or $2$ (resp. $0$,$1$,$2$, or $3$). Thus, a rank of $10$ (resp. $14$) in a hyperelliptic jacobian of genus $2$ (resp. genus $3$) is well above the average rank one would expect.

\begin{theorem}
	\label{main1}
	Let $g\geq 1$ be fixed, and assume that the jacobian of $\mathcal X$ over $\QQ(T)$ has no subvariety defined over $\QQ$ in its factorization. Then, Conjecture \ref{conj-nagao-new} implies that the jacobian of $\X$ has rank $4g+2$ over $\mathbb{Q}(T)$.
\end{theorem}

Each specialization of $T$ to an integer $t$ gives a hyperelliptic curve $\mathcal{X}_{t}$ over $\mathbb{Q}$ of genus $g$. By the specialization theorem of N\'eron, Silverman, and Tate we produce examples of hyperelliptic jacobians over $\QQ$ with moderate rank.
\begin{corollary}
	Let $g\geq 1$ be fixed, and assume the conditions from Theorem \ref{main1}. Then, there are infinitely many hyperelliptic curves over $\mathbb{Q}$ of genus $g$ with rank at least $4g+2$.
\end{corollary}
We remark here that Shioda \cite{Sh} has produced examples of hyperelliptic jacobians over $\QQ$ with rank $4g+7$. Our approach, however, is different from that of Shioda because we do not need to exhibit points in the jacobian $J_\X$ in order to (conjecturally) deduce its rank. Note also when $g=1$ we recover the result in \cite{ALM} and, in that case, the result is unconditional since Nagao's conjecture is known for rational surfaces.

\begin{example}\label{ex-intro}
	When $g=2$, our construction (see Sections \ref{sec-rank4gplus2} and \ref{sec-ex}) yields, for instance, the following hyperelliptic curve over $\QQ(T)$ of genus $2$, with trivial Chow trace (see below), and conjectural rank $10$:
	\begin{align*} 
 \X : y^2 &=	62476467927496043633049600000000x^5T^2 \\
 &+
	124952935854992087266099200000000x^5T \\
	& -
	3290807860845345873174084414821262950400000000x^5 \\
	& -
	78077124456852074329904550163688002129920000x^4T \\
	& +
	1266882949301025362537844681132821271997870080000x^4 \\
	& -
	123371083167607662332725955346616811520000000x^3T \\
	& +
	24393131657917882942419531475439645795721984020559648121x^3 \\
	& +
	97780947791238642428587970982523699200000000x^2T \\ 
	& +
	77106121667148850964656956255833136751214529427393152000x^2\\
	& -
	2549993916103702826374130630551142400000000000xT \\
	& -
	1078851918243051493072239063454153306319585738833920000x \\
	& +
	3290807860845408349642011910864896000000000000T \\
	& +
	1524014810925296267945145551729277974339657041182720000000.
	\end{align*} 
	When evaluating at $T=1$, we obtain a hyperelliptic curve over $\QQ$ of genus $2$ and rank $\geq 10$:
	\begin{align*}
		C: y^2 = f(x) =&  -3290807860845158443770301926690363801600000000x^5\\
		&  + 1266804872176568510463514776582657583995740160000x^4\\
		&  +
		24393131657794511859251923813106919840375367209039648121x^3\\
		&  +
		77106121667246631912448194898261724722197053126593152000x^2 \\
		& -
		1078851920793045409175941889828283936870728138833920000x\\
		&  +
		1524014810928587075805990960078919986250521937182720000000.
	\end{align*}
	First, we note that the quintic polynomial $f(x)=f(x,1)$ has Galois group $S_5$ (verified using Magma). Hence, the polynomial $f(x,T)$ must have generic Galois group $S_5$ as well. A theorem of Zarhin (Theorem \ref{thm-zarhin}) now shows that $J_\X(\QQ(T))$ must be an absolutely simple abelian variety. In particular, its Chow trace must be trivial.
	
In order to verify that the rank of the jacobian $J_C/\QQ$ of $C$ is $\geq 10$, we have found $10$ rational points $P_{\X,i}$ on  $\X(\QQ(T))$, given by:
\begin{align*}
   P_{\X,1} &= (1 , 7904205711360000T + 40304282819518226626505913739 ),\\
P_{\X,2} &=(4 , 252934582763520000T - 65685295896309228373784754088),\\
P_{\X,3} &=(9 , 1920721987860480000T - 159820218042093846301934185047),\\
P_{\X,4} &=(16 , 8093906648432640000T - 348080834252094356418710160704),\\
P_{\X,5} &=(25 , 24700642848000000000T - 656377570038206536713751374625),\\
P_{\X,6} &=(36 , 61463103611535360000T - 1113329506187909618252508872376),\\
P_{\X,7} &=(49 , 132845985390827520000T - 1748266818700297788803284771523),\\
P_{\X,8} &=(64 , 259005012749844480000T - 2590716593774949584327684741632),\\
P_{\X,9} &=(81 , 466735443050096640000T - 3670270185462998197299916564269),\\
P_{\X,10} &=(100 , 790420571136000000000T - 5016542911649893790058312437000).
\end{align*}
We evaluated $P_{\X,i}$ at $T=1$ to obtain points $P_i$ on $C(\QQ)$, and finally we defined points $(P_i)-(\mathcal{O})$ on the jacobian $J_C$, where $\mathcal{O}$ is the unique point at infinity. Here are the points:
\small 
\begin{align*} P_1 &= (1, 40304282819526130832217273739), P_2 = (4, 65685295896056293791021234088),\\
	P_3 &=  (9 , 159820218040173124314073705047), P_4 = (16 , 348080834244000449770277520704),\\
	P_5&= (25 : 656377570013505893865751374625 : 1),P_6= (36 : 1113329506126446514640973512376 : 1),\\
	P_7&=  (49 : 1748266818567451803412457251523 : 1),
	P_8= (64 : 2590716593515944571577840261632 : 1),\\
	P_9&=  (81 : 3670270184996262754249819924269 : 1),
	P_{10}= (100 : 5016542910859473218922312437000 : 1).
\end{align*}
\normalsize

Further, we have computed the canonical height matrix for these $10$ points, and its determinant, which equals $1131062371638072163.8139\ldots\neq 0$.  Thus, the points $\{(P_i)-(\mathcal{O})\}$ are independent and the rank of $J_C(\QQ)$ is at least $10$. Further, since these points come from evaluating points on $J_\X(\QQ(T))$, we conclude that the points $\{(P_{\X,i})-(\mathcal{O})\}$ on $J_\X$ must be independent also, and so the rank of $J_\X$ must be $\geq 10$ as well, in agreement with the generalized Nagao's Conjecture \ref{conj-nagao-new}.  
\end{example}
\normalsize

\subsection{Results on higher moments}

So far we have just focused on the first moments; however, the second moments are also interesting and play a key role in several problems. For one-parameter families of elliptic curves, Michel \cite{Mi} proved that if $j(T)$ is non-constant then \begin{equation} p A_{\mathcal{E},2}(p)  \ = \ p^2 + O\left(p^{3/2}\right) \end{equation} (we have multiplied the second moment by $p$ to match the quantity he studied); there are cohomological interpretations of the lower order terms, and Miller \cite{Mil3} showed that the bound is sharp by exhibiting a family with a term of that size.
While early investigations of the second moment was for the purpose of bounding the average rank in families, these are also key ingredients in understanding the behavior of zeros of the $L$-functions near the central point. According to the Katz-Sarnak theory \cite{KS1,KS2}, in the limit as the conductors tend to infinity the behavior of zeros near the central point behave similarly with the scaling limit of a subgroup of unitary matrices as their sizes tend to infinity (this is true for far more than just families of elliptic curves; see for example the survey article \cite{MMRTW}).

Interestingly, the main terms for the $n$-level densities of the low-lying zeros of families of elliptic curves depend only very weakly on the different values of the moments $A_{\mathcal{E},r}(p)$. The first two moments contribute to the main term; if the first moment is $k$ then there is a contribution equal to what one would expect if there were $k$ zeros at the central point, providing evidence in support of the Birch and Swinnerton-Dyer Conjecture (for more on this see \cite{GM}). The second moment's universality (of $p^2$ for families with non-constant $j(T)$, and $2p^2$ half the time and $0$ half the time for $j(T)$ constant) is similar to the universality of the second moments of Satake parameters in the work of Rudnick and Sarnak \cite{RuSa}, which was responsible for the universal behavior in the $n$-level correlations. If $r \ge 3$ then these terms never contribute to the main term. It is also similar to the Central Limit Theorem, where the universality is due to our ability to standardize any nice density to have mean zero and variance one (i.e., fix the first two moments), and the higher moments only surface in controlling the rate of convergence.

While the main term in the behavior of low-lying zeros is independent of the finer properties of the arithmetic of the curve, the higher moments ($r \ge 3$) and the lower order terms in the first and second moments \emph{are} observable in lower order corrections (i.e., in the rate of convergence). In particular, the lower order terms in the second moment have applications towards understanding the observed excess rank in many families of elliptic curves (see \cite{Mil3}), while the higher moments allow us to distinguish different families of elliptic curves in fine properties of the behavior of zeros near the central point (see \cite{Mil4}).

In this paper we concentrate on the first two moments. As has been remarked above, the first moment can be used to construct families with rank. The second moment is related to finer questions about the distribution of zeros. Interestingly, in all families of elliptic curves investigated to date, the first lower order term in the second moment $A_{\mathcal{E},2}(p)$ which does not average to zero always averages to a negative value; see \cite{A--,MMRW} for results for families of elliptic curves, as well as generalizations to other families of $L$-functions.  In our study of families of hyperelliptic curves, we observe this same bias. 

Our results are as follows. Let $g\geq 1$ be arbitrary, put $n=2g+1$, and let \(0 \leq k < n\). We shall consider the family \(\X_{n,h,k} : y^2 = x^n + x^h T^k\). We show the following (the proof can be found in Section \ref{sec-second}):

\begin{theorem}\label{thm-main-bias}
	Suppose \(\gcd(k,n-h,p-1) = 1\).
	Then
	\begin{align}
	p\cdot A_{\X_{n,h,1},2}(p)
	& = 
	\begin{cases}
	(\gcd(p-1,n-h) - 1)(p^2-p) & \text{ if } h ~ \textnormal{even},\\
	\gcd(n-h,p-1) (p^2-p)      & \text{ if } h ~ \textnormal{odd and} ~ \nu_2(p-1) > \nu_2(n-h),\\
	0                          & \textnormal{ otherwise}.
	\end{cases}
	\end{align}
\end{theorem}

The paper is organized as follows. In Section \ref{sec-HP} we recall the Hindry--Pacheco generalized version of Nagao's conjecture for surfaces and compare it to our version Conjecture \ref{conj-nagao-new}. In Section \ref{sec-prelim} we give some related results and some preliminary lemmas about computations with Legendre symbols. In Section \ref{sec-rank} we give constructions of hyperelliptic jacobians of rank $2g$, $2g+1$, and finally $4g+2$, in Sections \ref{sec-rank2g}, \ref{sec-rank2gplus1}, and \ref{sec-rank4gplus2}, respectively, and in Section \ref{sec-ex} we specialize our $4g+2$ construction in the case of $g=2$ and rank $10$. Finally, in Section \ref{sec-second}, we compute the second moments in a family of hyperelliptic curves of arbitrary genus.


\section*{Acknowledgements}
This research is from the Williams College SMALL REU Program organized by Steven J. Miller and was supported by Williams College and the National Science Foundation (grant number DMS1659037). The third named author was supported by the Finnerty Fund and thanks them for their support. The fifth named author was supported by the National Science Foundation (grant number DMS1561945). The authors would like to thank to Armand Brumer, Noam Elkies, and Joseph Silverman for helpful discussions.

\section{Generalized Nagao's conjecture for surfaces}\label{sec-HP}

In this section, we recall Hindry and Pacheco's \cite{HP} generalized Nagao's conjecture for a projective surface with a fibration onto a curve. Let $\X$ be a projective smooth irreducible surface over $\QQ$ with a proper flat fibration $f:\X\rightarrow \PP^{1}$ which allows us to have (arithmetic) curves of genus $g\geq 1$ for each fiber. Moreover, this implies that the generic fiber $X$ is a smooth irreducible curve over the function field of $\PP^{1}$, i.e., $\QQ(T)$. Denote by $J_{\X}$ the jacobian variety of $X$ and $(\tau,B)$ the $\QQ(T)/\QQ$-Chow trace of $J_{\X}$. The following theorem of Lang and N\'eron shows a Mordell--Weil analogue for a jacobian variety over function field.

\begin{theorem}
	The quotients $J_{\X}(\QQ(T))\big/\tau(B(\QQ))$ and $J_{\X}(\bar{\QQ}(T))\big/\tau(B(\bar{\QQ}))$	are both finitely generated groups.
\end{theorem}

Hindry and Pacheco generalized Nagao's conjecture to compute the rank of  $J_{\X}(\QQ(T))/\tau(B(\QQ))$. We describe Hindry and Pacheco's work next (over $\QQ$ for simplicity, but they work over a number field). 

Let $\X$ be a smooth irreducible projective surface over $\QQ$ and let $C$ be a smooth irreducible projective curve over $\QQ$ which allows a proper flat morphism $f\colon \X \to C$ so that the fibers are curves of (arithmetic) genus $g\geq 1$. Let $J_\X$ be the jacobian of $\X/\QQ(C)$, and let $(\tau,B)$ be the $\QQ(C)/\QQ$-Chow trace of $J_\X$.  For a prime $p$, we can consider the reduction $\tilde{f}\colon \tilde{\X}\to \tilde{C}$ of $f$. Define a finite set of primes $S$ which satisfies the following conditions: for all $p\not\in S$, the surface $\X$ and the curve $C$ have good reduction and the reduced morphism $\tilde{f}\colon \tilde{\X}\to \tilde{C}$ is proper and flat. And, the fibers are curves with arithmetic genus $g$ over the residue field $\Fp$. 

Denote $\tilde{\X}_t=\tilde{f}^{-1}(t)$, i.e., the fiber of $\tilde{f}$ at $t\in \tilde{C}(\Fp)$. Denote $\overline{\X}_t = \X_t \times_\QQ\overline{\QQ}$, where $\overline{\QQ}$ is a fixed algebraic closure of $\QQ$. Let $G_\QQ$ be the absolute Galois group of $\QQ$. Let $\Frob_p\in G_\QQ$ be a Frobenius element and $I_p\subset G_\QQ$ be the inertia group at $p$. Also, define the discriminant of the fibration $f$ by $\Delta_f = \{t\in C : \tilde{\X}_t \text{ is singular} \}.$ By enlarging the set $S$ if necessary, we can make the discriminant of $\tilde{f}$ be the same as the discriminant of $f$ modulo $p$ outside of $S$. 

Let $\overline{\Frob}_p$ be the Frobenius automorphism on $H_{\text{\'et}}^1(\overline{\X}_t,\QQ_\ell)$. Define the trace of Frobenius using cohomology as
$$a_{\X_t}(p) = \operatorname{Tr}( \overline{\Frob}_p|H_{\text{\'et}}^1(\overline{\X}_t,\QQ_\ell) ),$$
where we consider $\ell$-adic cohomology with compact support if $t\in \tilde{\Delta}_f(\Fp)$, that is to say,
\begin{align}\label{tracefrob} a_{\X_t}(p) = \operatorname{Tr}( \overline{\Frob}_p|H_{\text{c}}^1(\overline{\X}_t,\QQ_\ell) ).\end{align}
Also define a trace of Frobenius for the Chow trace $B$ by
$$a_B(p) = \operatorname{Tr}( \overline{\Frob}_p|H_{\text{\'et}}^1(\overline{B},\QQ_\ell)^{I_p} ),$$
where $\overline{B} = B \times_\QQ\overline{\QQ}$. By enlarging the set $S$ if necessary, we can assume that $B$ has good reduction for primes $p\not\in S$, i.e., 
$$a_B(p) = \operatorname{Tr}( \overline{\Frob}_p|H_{\text{\'et}}^1(\overline{B},\QQ_\ell) ).$$
Now we are ready to state Hindry and Pacheco's version of \cite{Na3}.
\begin{conjecture}[Hindry--Pacheco, \cite{HP}]\label{conj-HP}
	Let $\X$, $C$, and $f$ be as above, and define 
	$$A_{\X,1}(p)=\frac{1}{p}\sum_{t\in C(\Fp)}a_{\X_t}(p), \text{ and } A^\ast_\X(p)=A_{\X,1}(p) - a_B(p).$$
	Then:
	\begin{equation}
	\label{nagao}
	\lim_{P\to\infty}\frac{1}{P}\sum_{\substack{p \notin S \\ p \leq P}}-A^\ast_\X(p)\cdot \log p= \rank \left( \frac{J_{\X}(\QQ(C))}{\tau(B(\QQ))} \right).
	\end{equation}

\end{conjecture}

Under the hypothesis that the surface satisfies Tate's conjecture, they show the following.

\begin{theorem}[Hindry--Pacheco, \cite{HP}, Thm. 1.3] Suppose that the surface $\X$ satisfies Tate's conjecture. Then,
	$$\underset{s=1}{\operatorname{res}} \left(\sum_{p\not\in S}  -A^\ast_\X(p)\cdot \frac{\log p}{p^s} \right) = \rank \left( \frac{J_{\X}(\QQ(C))}{\tau(B(\QQ))} \right).$$
		If in addition the $L$ function $L_2(\X/\QQ,s)$ associated to $H_{\text{\'et}}^2(\overline{\X},\QQ_\ell)$ has an analytic continuation on $\Re(s)=2$, and does not have zeros on this line, then Conjecture \ref{conj-HP} holds.
\end{theorem}

In order to compare Conjecture \ref{conj-nagao-new} with the Hindry--Pacheco version,  we quote the following lemma from \cite[Lemme 3.2]{HP} (see also \cite[Lemma 1.7.]{RS}).

\begin{lemma}
	\label{trace} Let $\X/\QQ$  be a surface as above, with $C=\PP^1$.
	Let $p\notin S$, and $t\in \Fp$. Denote by $m_{t,p}$ the number of $\Fp$-rational components of the reduced fiber $\tilde{\X}_{t}$ modulo $p$. We denote by $\#\X_t(\Fp)$ the number of solutions on the reduced fiber $\tilde{\X}_{t}$ over $\Fp$. Let $a_{\X_t}(p)$ be the trace of Frobenius defined by Equation (\ref{tracefrob}). Then
	\begin{equation}
	\label{trace1}
	a_{\X_t}(p)=p\cdot m_{t,p}+1-\#\X_t(\Fp).
	\end{equation}
\end{lemma}

In particular, when $\X_{t}$ has good reduction at $p\notin S$, we get $m_{t,p}=1$ and therefore the definitions of $a_{\X_t}(p)$ in Equation (\ref{tracefrob}) and Conjecture \ref{conj-nagao-new} coincide. 

\section{Preliminaries and Auxiliary Lemmas}\label{sec-prelim}

First, we cite a result of Nagao, about the convergence of the limits that appear in the conjectures. Below, $\pi(x)$ is the prime counting function.

\begin{lemma}[Nagao, \cite{Na3}, Lemmas 2.1 and 2.2]\label{lem-nagaolimit} Let $\{c_p \}_p$ be a bounded sequence of non-negative numbers indexed by prime numbers. If one of the sequences of numbers
	$$\left\{\frac{1}{\pi(N)} \sum_{p\leq N} c_p  \right\}_N \text{ and } \left\{\frac{1}{N} \sum_{p\leq N} c_p\cdot \log p  \right\}_N$$
	converges, then both of them converge to a common limit. In particular,  $\{\frac{1}{N}\sum_{p\leq N} \log p\}_N$ is a  convergent sequence, and the limit is $1$.
\end{lemma}

Next we cite work of Zarhin \cite{Za1,Za2,Za3} that we shall use in order to prove that a Chow trace is trivial.

\begin{theorem}[Zarhin]\label{thm-zarhin} Let $K$ be a field of characteristic different from $2$, and suppose that $f\in K[x]$ is a polynomial of degree $n$ without multiple roots, such that $\Gal(f)$ is either $S_n$ or $A_n$. Let $C: y^2=f(x)$ and let $J(C_f)$ be its jacobian. Assume also that either $\operatorname{char}(K)\neq 3$ or $n\geq 7$. Then, $\operatorname{End}(J(C_f))=\ZZ$. In particular, $J(C_f)$ is an absolutely simple abelian variety. 
\end{theorem}

Let $C: y^2=f(x)$ be a hyperelliptic curve, with $f(x)$ of odd degree, and let $p$ be a prime number. Since the degree of $f$ is odd, there is a unique point $\mathcal{O}$ at infinity. Then, one can count the (affine) points on $C^\ast(\Fp)=C(\Fp)-\{\mathcal{O} \}$ via Legendre symbols, by
\begin{align*} \# C^\ast(\Fp) &= \sum_{\substack{x(p) \\ f(x)\equiv 0 \bmod p}}1 + \sum_{\substack{x(p) \\ f(x)\not\equiv 0 \bmod p}} \left( 1 + \legendre{f(x)}{p}\right) =p+ \sum_{\substack{x(p)}} \legendre{f(x)}{p},\end{align*}
where $x(p)=\{ x : 0\leq x \leq p-1\}$. Thus, $$a_C(p) = p+1 - \# C(\Fp) = p+1 - \left(1 + p+ \sum_{\substack{x(p) }} \legendre{f(x)}{p}\right)=-\sum_{\substack{x(p) }} \legendre{f(x)}{p},$$
 Thus, if $\X: y^2=f(x,T)$ is a hyperelliptic surface, then
$$A_{\X,1}(p) = \frac{1}{p}\sum_{t(p)} a_{\X_t}(p) = -\frac{1}{p} \sum_{\substack{t,x(p)  }} \legendre{f(x,t)}{p},$$
where $t,x(p)$ means that both $t$ and $x$ range in the interval $0,\ldots,p-1$. In the remainder of this section, we show a number of lemmas about sums of Legendre symbols that we will use in the next sections. First, we reproduce \cite[Lemma A.2]{ALM} that shall be used in computing first and second moments.
\begin{lemma}\label{AMLRlemma}
	Assume $a$ and $b$ are not both zero $\bmod p$ and $p > 2$. Then
	\begin{equation}
	\sum_{t=0}^{p-1} \legendre{at^2 + bt + c}{p} = 
	\begin{cases}
	(p-1)\legendre{a}{p} & \ \mathrm{if} \ p \mid b^2 - 4ac, \\
	-\legendre{a}{p} & \ \mathrm{otherwise.}
	\end{cases}
	\end{equation}
\end{lemma}

\begin{lemma}\label{lem-addstozero} For any $p>2$, we have $\sum_{t(p)} \legendre{at+b}{p} = 0$, for any integers $a,b\in\Z$ with $a\not\equiv 0 \bmod p$.
\end{lemma}
\begin{proof}
	As $t$ runs over all elements of $\Fp$, the quantity $at+b$ also runs over all elements of $\Fp$. There are $(p-1)/2$ quadratic residues, and $(p-1)/2$ quadratic non-residues, so the sum cancels out. 
\end{proof}

\begin{lemma}
	\label{power}
	Fix \(n\geq 1\) and a prime number \(p\).
	Then, the number of solutions $(x,y)$ to the congruence \(x^n \equiv y^n \bmod p\) is given by \(\gcd(p-1,n)\cdot (p-1) + 1\).
\end{lemma}
\begin{proof}
	First note that $x \equiv y \equiv 0 \bmod p$ is a solution. Otherwise, let $x,y$ be non-zero mod $p$, with $x^n\equiv y^n \bmod p$. Let $g$ be a generator of $\Fp^\times$, and write $x \equiv g^a$ and $y \equiv g^b \bmod p$ for some $a$ and $b$. Hence we have 
	$g^{an} \equiv g^{bn} \bmod p$ and so $g^{an-bn} \equiv 1 \bmod p.$
	Since the order of $g$ is $p-1$, by Lagrange's theorem, we have
	$$an - bn \equiv n(a-b) \equiv 0 \bmod {p-1}.$$ 
	The congruence $nX \equiv 0 \bmod p-1$ has $\gcd(n,p-1)$ solutions for $X$ and each solution for $X$ yields $p-1$ solutions $(a,b)$. Hence, there are $\gcd(n,p-1)\cdot (p-1) + 1$ solutions to the original congruence. 
\end{proof}

\newcommand{\mindicator}{\textbf{*1}}

\begin{lemma}
	\label{double_sums}
	Let \(h \geq 2\) be even.
	Let \(\displaystyle S_h(p) := \sum_{\substack{x,y(p) \\ x^h \equiv y^h \bmod p}} \legendre{xy}{p}\).
	Furthermore, let \(\nu_2\) be the usual \(2\)-adic valuation.
	Then,
	\begin{equation}
	S_h(p) =
	\begin{cases}
	\gcd(h,p-1) (p-1) & \text{if }\nu_2(p-1) > \nu_2(h),\\
	0                 & \textnormal{otherwise}.
	\end{cases}
	\end{equation}
\end{lemma}
\begin{proof}
	We proceed by passing from the multiplicative group \(\mathbb{F}_p^\times\) to the cyclic additive group \(\Z/(p-1)\Z\) (by implicitly fixing a primitive root).
	In \(\Z/(p-1)\), we use \(\langle a \rangle\) to denote the subgroup generated by \(a\).
	\begin{align}
	S_h(p)
	&:= \sum_{\substack{a,b(p) \\ a^h \equiv b^h \bmod p}} \legendre{ab}{p} =  \sum_{\substack{x,y(p-1) \\ ha \equiv hb\bmod p-1}} \left(
	\begin{cases}
	1  & a+b \in \langle 2 \rangle\\
	-1 & \textnormal{otherwise}
	\end{cases}
	\right).
	\end{align}
	We know that \(h(y-x) \equiv 0 \bmod p\) iff \(y-x \in \left\langle \frac{p-1}{\gcd(h,p-1)} \right\rangle\) iff \(y \in x + \left\langle \frac{p-1}{\gcd(h,p-1)} \right \rangle\).
	Furthermore, \(x + y \in \langle 2 \rangle\) iff \(x + y - 2x \in \langle 2 \rangle\) iff \(y \in x + \langle 2 \rangle\).
	Thus, we have
	\begin{align}
	&= \sum_{\substack{x,y(p-1) \\ y \in \left\langle \frac{p-1}{\gcd(h,p-1)} \right\rangle}} \left(
	\begin{cases}
	1  & y \in x + \langle 2 \rangle\\
	-1 & \textnormal{otherwise}
	\end{cases}
	\right) = \sum_{x(p-1)} \sum_{\substack{y(p-1) \\ y \in x + \left\langle \frac{p-1}{\gcd(h,p-1)} \right\rangle}} (2 (\textbf{1}_{y \in x + \langle 2 \rangle}) - 1)\\
	&= \sum_{x(p-1)} \left( 2 \left| \left\langle \frac{p-1}{\gcd(h,p-1)} \right\rangle \cap \langle 2 \rangle \right| - \left| \left\langle \frac{p-1}{\gcd(h,p-1)} \right\rangle \right| \right).
	\end{align}
	If \(\frac{p-1}{\gcd(h,p-1)} \in \langle 2 \rangle\), then this is
	\begin{align}
	&= \sum_{x(p-1)} \left( 2 \left| \left\langle \frac{p-1}{\gcd(h,p-1)} \right\rangle \right| - \left| \left\langle \frac{p-1}{\gcd(h,p-1)} \right\rangle \right| \right) = (p-1) \gcd(h,p-1).
	\end{align}
	If \(\frac{p-1}{\gcd(h,p-1)} \not \in \langle 2 \rangle\), then this is
	\begin{align}
	&= \sum_{x(p-1)} \left( 2 \cdot \frac{1}{2} \left| \left\langle \frac{p-1}{\gcd(h,p-1)} \right\rangle \right| - \left| \left\langle \frac{p-1}{\gcd(h,p-1)} \right\rangle \right| \right) = 0.
	\end{align}
	Since \(\frac{p-1}{\gcd(h,p-1)} \in \langle 2 \rangle\) precisely when \(\nu_2(p-1) > \nu_2(h)\), this completes the proof.
\end{proof}

\begin{lemma}
	\label{gcd-lemma}
	Let \(k\), \(n_1\), and \(n_2\) be integers.
	If \(\gcd(k,n_1,n_2) = 1\), then there exists an \(m\) such that \(\gcd(m,n_1) = 1\) and \(m \equiv k \bmod n_2\).
\end{lemma}
\begin{proof}
	By Dirichlet's theorem on primes in arithmetic progressions, one can choose a prime $m\equiv k \bmod n_2$ such that $m>n_1$. 
\end{proof}

\section{Jacobians of Hyperelliptic Curves with Positive Rank Over \(\QQ(T)\)} \label{sec-rank}

In this section we will show a number of constructions that yield hyperelliptic curves $\X/\QQ(T)$ such that their jacobians have positive rank.

\subsection{Rank $2g$.} \label{sec-rank2g}

In \cite[Proposition 3.2.]{Na3}, Nagao considers the  elliptic curve 
$$\E/\QQ(T): y^{2}=f(x)+T^{2},$$
where $f(x)\in\ZZ[x]$ is a monic cubic polynomial without double roots, and computed $A_{\E,1}(p)$ for this family. In particular, he showed
\begin{equation}
-A_{\E,1}(p) = \begin{cases} 
2p & \text{if} \ f(x) ~\text{mod}~ p~ \text{factors completely,} \\
-p & \text{if} \  f(x) ~\text{mod}~ p~ \text{is irreducible,} \\
0 & \text{if} \ ~\text{otherwise.} \\
\end{cases}.
\end{equation}
We now present an analogous construction for curves of genus $g$ and compute their corresponding Legendre sums. Let $g\geq 1$ be fixed, and let 
$$\X/\QQ(T): y^2 = f(x) + T^2,$$ 
where $f(x)$ is a monic degree $2g+1$ polynomial without double roots. Note that the discriminant in $t$ of $f(x) + t^2$ is $-4f(x)$, and so if $p > 2$, then the condition that $p \mid -4f(x)$ is equivalent to $p \mid f(x)$. Hence, using Lemma \ref{AMLRlemma}, we have
\begin{align}
\begin{split}
-p\cdot A_{\X,1}(p) &= \sum_{t,x(p)} \legendre{f(x) + t^2}{p} \\
&= \sum_{\substack{x(p) \\ f(x) \equiv 0 \bmod p}} (p-1)\legendre{1}{p} - \sum_{\substack{x(p) \\ f(x) \nequiv 0\bmod p}} \legendre{1}{p} \\
&= \sum_{\substack{x(p) \\ f(x) \equiv 0}} p\legendre{1}{p} - \sum_{\substack{x(p) \\ f(x) \equiv 0}} \legendre{1}{p} - \sum_{\substack{x(p) \\ f(x) \nequiv 0}} \legendre{1}{p} \\
&= \sum_{\substack{x(p) \\ f(x) \equiv 0}} p - \sum_{x(p)} \legendre{1}{p} = p\#\{x \in \ZZ/p\ZZ \mid f(x) \equiv 0 ~\text{mod}~ p\} - p.
\end{split}
\end{align}
Let $\tilde{f}(x)$ denote $f(x) \bmod p$ and write $L_f$ for the number of distinct linear factors in the factorization of $\tilde{f}(x)$. Then
\begin{equation}
-p\cdot A_{\X,1}(p) = (L_f-1)\cdot p,
\end{equation}
and in particular, if $\tilde{f}(x)$ factors completely, then $L_f = 2g+1$ and $-p\cdot A_{\X,1}(p) = 2gp$. If we choose $f(x)$ so that it splits completely over $\QQ$, then we have $-p\cdot A_{\X,1}(p) = 2gp$ for all sufficiently large $p$.  Thus, Lemma \ref{lem-nagaolimit} implies that
$$\lim_{P \to \infty} \frac{1}{P} \sum_{p\leq P} -A_{\X,1}(p)\cdot \log  p = \lim_{P \to \infty} \frac{1}{P} \sum_{p\leq P} 2g\cdot \log  p =2g.$$
Further, if $f(x)$ is a Morse function, then the Galois group of $f(x)+T^2$ over $\QQ(T)$ is $S_n$, by \cite[Theorem 4.4.5]{Se}. Hence, Zarhin's Theorem \ref{thm-zarhin} shows that, under these conditions, $J_\X(\QQ(T))$ is an absolutely simple abelian variety, and therefore the $\QQ(T)/\QQ$-Chow trace of $\X$ is trivial. Thus, Conjecture \ref{conj-nagao-new} implies that $\rank J_\X(\QQ(T)) = 2g.$ 

In this case, Shioda \cite{Sh2} has shown that $J_\X(\QQ(T))$ has rank $2g$ by exhibiting $2g$ independent points. Moreover, Hindry and Pacheco \cite[Exemple 5.6]{HP} show that $\X$ is a rational surface, and therefore Tate's conjecture and Conjecture \ref{conj-HP} hold for this surface, showing again that $\rank J_\X(\QQ(T)) = 2g.$

\begin{example}
	Let $f(x) = (x-1)(x-2)\cdots(x-7)$ and let $\X: y^2 = f(x)+T^2$. Specializing at $T=2$ we obtain the hyperelliptic curve of genus $3$ given by
	$$\X_2 : y^2 = x^7 - 28x^6 + 322x^5 - 1960x^4 + 6769x^3 - 13132x^2 + 13068x - 5036,$$
	which Magma shows to be of rank $2g=6$. In fact, one can check that the points $(P_i)-(\mathcal{O})$ given by $P_i=(i,T)$, for $T=2$, are independent in $J_{\X_2}(\QQ)$. Hence, the same points are independent in $J_\X$ over $\QQ(T)$. Magma also verifies that the Galois group of $f(x)+4$ is $S_7$, and therefore the Galois group of $f(x)+T^2$ over $\QQ(T)$ is $S_7$ as well. Thus, $J_\X(\QQ(T))$ is a simple hyperelliptic jacobian of rank $6$.
	
	If we put $f(x)= (x-1)(x-2)\cdots(x-7)$ and let $\X: y^2 = f(x)+T^2$, then Magma verifies that the specialization at $T=2$ yields a curve of genus $4$ given by
\small 	
	$$y^2 = x^9 - 45x^8 + 870x^7 - 9450x^6 + 63273x^5 - 269325x^4 + 723680x^3 -
	1172700x^2 + 1026576x - 362876$$
	with rank $9$. Note that here $J_\X(\QQ(T))$ is of rank $2g=8$, but the specialization at $T=2$ has higher rank equal to $9$.
\end{example}

\subsection{Rank $2g+1$.}\label{sec-rank2gplus1}

We consider the family \(\X: y^2 = f(x)T + 1\), where $f(x)$ is a polynomial of degree $2g+1$ that splits completely, with no double roots. We compute the first moment of \(\X\) as follows:
\begin{align}
-p\cdot A_{1,\X}
&=
-\sum_{t(p)} \sum_{x(p)} \legendre{f(x)t + 1}{p}\\
&=
-\sum_{\substack{x(p) \\ f(x) \equiv 0\bmod p}} \sum_{t(p)} \legendre{1}{p} - \sum_{\substack{x(p) \\ f(x) \not\equiv 0\bmod p}} \sum_{t(p)} \legendre{f(x)t + 1}{p}\\
&=
-p\sum_{\substack{x(p) \\ f(x) \equiv 0\bmod p}} 1 + 0=-L_f\cdot p,
\end{align}
where we have used Lemma \ref{lem-addstozero}, and $L_f$ is the number of linear factors in the factorization of $\tilde{f}(x)$ modulo $p$, as before. Since $f(x)$ factors completely over $\QQ$, it follows that $L_f=2g+1$ for all sufficiently large $p$. Thus, $-A_{\X,1}(p)=2g+1$, and Lemma \ref{lem-nagaolimit} implies that
$$\lim_{P \to \infty} \frac{1}{P} \sum_{p\leq P} -A_{\X,1}(p)\cdot \log  p =2g+1.$$
Thus, if we assume that $f(x)$ is chosen so that the $\QQ(T)/\QQ$-Chow trace of $\X$ is trivial, then Conjecture \ref{conj-nagao-new} implies that the rank of $J_\X(\QQ(T))$ is $2g+1$.

\begin{example}
	Let $f(x) = (x-1)(x-2)\cdots(x-7)$ and let $\X: y^2 = f(x)T+1$. Specializing at $T=2$ we obtain the hyperelliptic curve of genus $3$ given by
	$$\X_2 : y^2 =2x^7 - 56x^6 + 644x^5 - 3920x^4 + 13538x^3 - 26264x^2 + 26136x - 10079,$$
	which Magma shows to be of rank $2g+1=7$. In fact, one can check that the points $(P_i)-(\mathcal{O})$ given by $P_i=(i,1)$, for $T=2$, are independent in $J_{\X_2}(\QQ)$. Hence, the same points are independent in $J_\X$ over $\QQ(T)$. Magma also verifies that the Galois group of $2f(x)+1$ is $S_7$, and therefore the Galois group of $f(x)T+1$ over $\QQ(T)$ is $S_7$ as well. Thus, $J_\X(\QQ(T))$ is a simple hyperelliptic jacobian of rank $7$.
\end{example}

\subsection{Rank $4g+2$.}\label{sec-rank4gplus2}

Consider the following genus \(g\) curve:
\begin{align}
\X: y^2 = f(x,T)
&= x^{2g+1} T^2 + 2q(x)T - h(x)
\end{align}
where $q(x)$ and $h(x)$ are polynomials in $\Z[x]$ of degree $2g+1$ to be chosen at a later time. We will imitate the construction in \cite{ALM} for $g=1$, and choose \(q(x)\) and \(h(x)\) such that \(\text{rank}(\X) = 4g+2\). In order to do so, we define $D_T(x)$ as a fourth of the discriminant of $f(x,T)$ as a polynomial in the variable $T$.


\begin{lemma}\label{lem-momentbigrank}
	Suppose that \(D_T(x)\) has distinct integer roots \(r_i = \rho_i^2\) for \(1 \leq i \leq 4g+2\).
	Then, the first moment of the surface $\X$ satisfies $-p\cdot A_{\X,1}(p) = (4g+2)p$, for all sufficiently large $p$.
\end{lemma}
\begin{proof}
	We use Lemma \ref{AMLRlemma} to compute 
	\begin{align}
	-p\cdot A_{\X,1}(p)
	&=
	\sum_{t(p)} \sum_{x(p)} \legendre{f(x,t)}{p} =
	\sum_{t(p)} \sum_{x(p)} \legendre{x^{2g+1} T^2 + 2q(x)T - h(x)}{p}\\
	&=
	\sum_{\substack{x(p) \\ D_t(x) \equiv 0\bmod p}}(p-1) \legendre{x^{2g+1}}{p} - \sum_{\substack{x(p) \\ D_t(x) \not \equiv 0\bmod p}}\legendre{x^{2g+1}}{p}\\
	&=
	\sum_{\substack{x(p) \\  D_t(x) \equiv 0\bmod p}}p \legendre{x}{p} - \sum_{\substack{x(p)}} \legendre{x}{p} =
	p\cdot\left( \sum_{\substack{x(p) \\  D_t(x) \equiv 0\bmod p}} \legendre{x}{p} \right). 
	\end{align}
	Since $D_T(x)$ has $4g+2$ distinct integer roots, for large enough $p$ these will also be distinct modulo $p$, and therefore $-p\cdot A_{\X,1}(p) = (4g+2)p$, as claimed.
\end{proof}

\begin{lemma}\label{lem-choosehq}
	For any choice of distinct integers $\rho_1,\ldots,\rho_{4g+2}$, there exists polynomials  
	$$q(x)   =  x^{2g+1} + \sum_{i=0}^{2g} a_i x^i, \text{ and } h(x)   = (A-1)x^{2g+1} + \sum_{i=0}^{2g} A_ix^i,$$
	in $\mathbb{Z}[x]$ so that $D_T(x)=q(x)^2 + x^{2g+1}h(x)$ has $4g+2$ distinct roots $r_i = \rho_i^2$.
\end{lemma}
\begin{proof}
	We want to choose \(a_i\) and \(A_i\) such that
	\begin{align*}
	D_T(x) &= q(x)^2 + x^{2g+1}h(x) \\
	&= Ax^{4g+2} + \left(A_{2g} + 2a_{2g}\right)x^{4g+1} + \left(A_{2g-1} + a_{2g}^2 + 2a_{2g-1}\right)x^{4g} + \dots + (2a_0a_1)x + a_0^2 \\
	&= A\left(x^{4g+2} + R_{4g+1}x^{4g+1} + R_{4g}x^{4g} + \dots + R_{1}x + R_{0}\right) \\
	&= A(x-\rho_1^2)(x-\rho_2^2)\cdots(x-\rho_{4g+1}^2)(x-\rho_{4g+2}^2).
	\end{align*}
	The only nontrivial equality here is between the second and third lines.
	We choose \(a_i\) and \(A_i\) by equating coefficients between those two lines.
	Thus, we have \(4g+1\) equations in \(4g+3\) variables.
	In particular, the equation of the \(k\)th coefficient is the following (where \((\cdot)_k\) denotes the coefficient of \(x^k\)):
	\begin{equation*}
	R_kA = (D_T(x))_k,
	\end{equation*}
	for \(0 \leq k \leq 4g+2\). Note that $R_{4g+2}=1$. Furthermore, we have
	\begin{align*}
	(D_T(x))_k
	&= (q(x)^2)_k + (h(x)x^{2g+1})_k\\
	&= (h(x)x^{2g+1})_k + \sum_{\substack{i+j=k \\ 0 \leq i,j \leq k}} (q(x))_i (q(x))_j  \\
	&= A_{k-(2g+1)} + \sum_{\substack{i+j=k \\ 0 \leq i,j \leq k}} a_ia_j,
	\end{align*}
	where \(A_i = 0\) by convention if \(i < 0\).
	Therefore, it suffices to solve for the variable $R_kA$ via
	\begin{equation}
	\label{eq88}
	R_kA = A_{k-(2g+1)} + \sum_{\substack{i+j=k \\ 0 \leq i,j \leq k}} a_ia_j  
	\end{equation}
	for \(0 \leq k \leq 4g+2\). For consistency, we require $A_{2g+1} = A$ and $a_{2g+1} = 1$. Recall that $R_{k}$ are fixed for $0\leq k \leq 4g+2$. First of all, for $2g+1 \leq k \leq 4g+2$, we can choose the adequate coefficients of $h(x)$ to satisfy \eqref{eq88}. Thus, if we list these equations for $0\leq k \leq 2g$ as \(k\) increases, we should determine the coefficients of $q(x)$. In the \(k = 0\) case, the equation is \(R_0A = a_0^2\).
	Thus, if we allow rational solutions and choose \(A\) such that \(R_0A\) is a square in \(\Z\), then these equations have solutions.
	
	This yields \(h,q \in \QQ[x]\) such that the roots of \(D_T(x)\) meet the stated condition.
	However, we require \(h,q \in \Z[x]\).
	To move \(h\) and \(g\) into \(\Z[x]\) without changing the roots of \(D_T(x) = q(x)^2+h(x)x^{2g+1}\), we replace \(q(x)\) with \(L\cdot q(x)\) and \(h(x)\) with \(L^2\cdot h(x)\) where \(L\) is the least common multiple of the denominators of the coefficients of \(h\) and \(q\).
\end{proof}

We are now ready to restate and prove our main Theorem \ref{main1}.

\begin{theorem}
	\label{thm-main1proof}
	Let $g\geq 1$ be fixed, and let $\rho_1,\ldots,\rho_{4g+2}$ be distinct integers. Let $h(x)$ and $q(x)$ be chosen as in Lemma \ref{lem-choosehq}, and let 
	$$\X: y^2 = x^{2g+1} T^2 + 2q(x)T - h(x).$$
	Assume that the jacobian of $\mathcal X$ over $\QQ(T)$ has no subvariety defined over $\QQ$ in its factorization. Then, Conjecture \ref{conj-nagao-new} implies that the jacobian of $\X$ has rank $4g+2$ over $\mathbb{Q}(T)$.
\end{theorem}  

\begin{proof}
	Let $g\geq 1$ be fixed, and let $\rho_1,\ldots,\rho_{4g+2}$ be distinct integers. By Lemma \ref{lem-choosehq}, we can find polynomials $h(x)$ and $q(x)$ such that $D_T(x)=(q(x))^2+x^{2g+1}h(x)=A(x-\rho_1^2)\cdots(x-\rho_{4g+2}^2)$. If we define $\X$ as in the statement, then Lemma \ref{lem-momentbigrank} shows that $-p\cdot A_{\X,1}(p) = (4g+2)p$ for all sufficiently large primes $p$. Therefore, Lemma \ref{lem-nagaolimit} shows that 
$$	\lim_{P \to \infty} \frac{1}{P} \sum_{p \leq P} -A_{\X,1}(p) \log p\ = 4g+2.$$
Moreover, if we assume that $J_\X$ has no subvariety defined over $\QQ$, then its Chow trace must be trivial. Thus, Conjecture \ref{conj-nagao-new} implies that the rank of $J_\X(\QQ(T))$ is $4g+2$, as claimed.
\end{proof}

\begin{remark} If desired, we can change variables so that $\X$ is given in the form $y^2=F(x,T)$ with $F(x,T)$ monic in the variable $x$. Indeed, we replace $y^2=f(x,T)$ by the form $y^2=F(x,T)$ with a monic polynomial $F(x,T)$ by changing the basis
	$$y \rightarrow \frac{y}{(T^2+2T-A+1)^2},~~~x\rightarrow \frac{x}{T^2+2T-A+1}.$$
	Then we have
	\small 
	\begin{align*}
	\begin{split}
	y^2 &=f(x,T)\\
	&= x^5T^2 + 2(x^5+ax^4+bx^3+cx^2+dx+e)T-(A-1)x^5-Bx^4-Cx^3-Dx^2-Ex-F\\
	&= (T^2+2T-A+1)x^5+(2aT-B)x^4+(2bT-C)x^3+(2cT-D)x^2+(2dT-E)x+(2eT-F)\\
	y^2 &=F(x,T)\\
	&= x^5+(2aT-B)x^4+(2bT-C)(T^2+2T-A+1)x^3+(2cT-D)(T^2+2T-A+1)^2 x^2\\
	&~~~~+ (2dT-E)(T^2+2T-A+1)^3 x+(2eT-F)(T^2+2T-A+1)^4.
	\end{split}
	\end{align*}
	\normalsize
\end{remark}

\subsection{An example in genus $2$ and rank $10$.}\label{sec-ex}

In this section, we follow the recipe of Theorem \ref{thm-main1proof} to construct a hyperelliptic curve $\X/\QQ(T)$ with jacobian of rank $10$, which is how we found the curve in Example \ref{ex-intro} of the introduction. Let $\rho_1,\ldots,\rho_{10}$ be distinct integers.  We need to find polynomials 
\begin{align*}
q(x) = x^5+ax^4+bx^3+cx^2+dx+e \text{ and } h(x) = (A-1)x^5+Bx^4+Cx^3+Dx^2+Ex+F,
\end{align*}
such that 
\begin{align}
\label{simul}
\begin{split}
D(x) & =q(x)^2+x^5h(x) \\
&=Ax^{10}+(B+2a)x^9+(C+a^2+2b)x^8+(D+2ab+2c)x^7\\
&~~~ +(E+2d+2ac+b^2)x^6+(F+2e+2ad+2bc)x^5\\
&~~~+(2ae+2bd+c^2)x^4+(2be+2cd)x^3+(2ce+d^2)x^2+(2de)x+e^2.\\
&=A(x^{10}+R_{9}x^9+R_{8}x^8+R_{7}x^{7}+R_{6}x^{6}+R_{5}x^{5}+R_{4}x^{4}+R_{3}x^{3}+R_{2}x^2+R_{1}x+R_{0})\\
&=A(x-\rho_{1}^{2})(x-\rho_{2}^{2})(x-\rho_{3}^{2})(x-\rho_{4}^{2})(x-\rho_{5}^{2})(x-\rho_{6}^{2})(x-\rho_{7}^{2})(x-\rho_{8}^{2})(x-\rho_{9}^{2})(x-\rho_{10}^{2}).
\end{split}
\end{align}

Now we will explicitly describe how to determine coefficients of $q(x)$ and $h(x)$ for given $R_{i}$, $i=0,1,\ldots, 9$. Since we can adjust the integer values of $B,C,D,E, \text{and}~ F$, solving the following simultaneous equations from \eqref{simul} is equivalent to give $q(x),h(x)\in\ZZ[x]$ for any given distinct roots of $D_{T}(x)$.
\begin{align*}
2ae+2bd+c^2 & =R_{4}A,\\
2be+2cd &= R_{3}A,\\
2ce+d^{2} &= R_{2}A,\\
2de &= R_{1}A,\\
e^{2} &= R_{0}A.
\end{align*}
For simplifying the procedure, let $A=2e$. Then, we can find the values for the rest of the constants recursively as follows:
\begin{align}
\label{simul2}
\begin{split}
e & = 2R_0,\ d = R_1,\ c = (2eR_2-d^2)/2e,\ b = (2eR_3-2cd)/2e,\ a = (2eR_4-c^2-2bd)/2e,\\
F&=R_{5}A-2e-2ad-2bc,\ E=R_{6}A-2d-2ac-b^2,\ D=R_{7}A-2ab-2c,\\ C&=R_{8}A-a^2-2b,\ B=R_{9}A-2a,\ A=4R_0=2e.
\end{split}
\end{align}
Note: these coefficients are not in $\Z$ but later it is easy to find an integral model for $\X$.

\begin{example}
	For simplicity, let $\rho_{i}=i$, for $i=1,\ldots, 10$. Then we get
	$$R_{0}=13168189440000,~R_{1}=-20407635072000,~R_{2}=8689315795776,$$
	$$R_{3}=-1593719752240,~R_{4}=151847872396, ~R_{5}=-8261931405,$$
	\begin{equation}
	R_{6}=268880381,~R_{7}=-5293970,~R_{8}=61446,~R_{9}=-385.
	\end{equation}
From Eq. (\ref{simul2}), we obtain the coefficients $a,\ldots,e$, and then $A,\ldots,F$, and we build a hyperelliptic curve:
	\begin{align*} 
	\X : y^2 &=	62476467927496043633049600000000x^5T^2 \\
	&+
	124952935854992087266099200000000x^5T \\
	& -
	3290807860845345873174084414821262950400000000x^5 \\
	& -
	78077124456852074329904550163688002129920000x^4T \\
	& +
	1266882949301025362537844681132821271997870080000x^4 \\
	& -
	123371083167607662332725955346616811520000000x^3T \\
	& +
	24393131657917882942419531475439645795721984020559648121x^3 \\
	& +
	97780947791238642428587970982523699200000000x^2T \\ 
	& +
	77106121667148850964656956255833136751214529427393152000x^2\\
	& -
	2549993916103702826374130630551142400000000000xT \\
	& -
	1078851918243051493072239063454153306319585738833920000x \\
	& +
	3290807860845408349642011910864896000000000000T \\
	& +
	1524014810925296267945145551729277974339657041182720000000.
	\end{align*} 
Then, we can proceed as in Example \ref{ex-intro} to show that $\rank J_\X(\QQ(T))=10$, unconditionally, and therefore verifying Conjecture \ref{conj-nagao-new} in this case.
\end{example}

\section{The Second Moments}\label{sec-second}

In this section we compute the second moments in a family of hyperelliptic curves of the form
$$\X_{n,h,k} : y^2 = x^n + x^h T^k$$ where \(n = 2g+1\), and \(g\) is the genus of \(\X\), and  \(0 \leq k < n\).

We first write a formula for \(A_{\X,2}(p)=\frac{1}{p}\sum_{t(p)} a_{\X_{t}}(p)^2\) in a useful way.
For convenience, let \(c\) be \(2\) if \(h\) is even and \(1\) if \(h\) is odd.
Then we have
\begin{align}
\begin{split}
p\cdot A_{\X_{n,h,k},2}(p)& = \sum_{t(p)} a_{\X_{t}}(p)^2 = \sum_{t(p)} \left(-\sum_{x(p)} \legendre{x^n+x^ht^k}{p} \right)^2 \\
&=
\sum_{t(p)} \sum_{x(p)} \sum_{y(p)} \legendre{x^n + x^h t^k}{p} \legendre{y^n + y^h t^k}{p}\\
&=
\sum_{t = 1}^{p-1} \sum_{x(p)} \sum_{y(p)} \legendre{x^n + x^h t^k}{p} \legendre{y^n + y^h t^k}{p}\\
&=
\sum_{t = 1}^{p-1} \sum_{x(p)} \sum_{y(p)} \legendre{x^hy^h}{p} \legendre{x^{n-h} + t^k}{p} \legendre{y^{n-h} + t^k}{p}\\
&=
\sum_{t = 1}^{p-1} \sum_{x(p)} \sum_{y(p)} \legendre{xy}{p}^c \legendre{x^{n-h} + t^k}{p} \legendre{y^{n-h} + t^k}{p}.
\end{split}
\end{align}
where we have used Lemma \ref{lem-addstozero} to remove $t=0$ from the summations. This formula is effectively casewise on the parity of \(h\): if \(h\) is odd, there is an \(\legendre{xy}{p}\) term, and if \(h\) is even, there is not an \(\legendre{xy}{p}\) term. We first use this formula to prove a \(k\)-periodicity result:
\begin{align}
\begin{split}
p\cdot A_{\X_{n,h,k},2}(p)
&=
\sum_{t = 1}^{p-1} \sum_{x(p)} \sum_{y(p)} \legendre{xy}{p}^c \legendre{x^{n-h} + t^k}{p} \legendre{y^{n-h} + t^k}{p}\\
&=
\sum_{t = 1}^{p-1} \sum_{x(p)} \sum_{y(p)} \legendre{(t^{-1} x)(t^{-1} y)}{p}^c \legendre{t^{h-n} x^{n-h} + t^k}{p} \legendre{t^{h-n} y^{n-h} + t^k}{p}\\
&=
\sum_{t = 1}^{p-1} \sum_{x(p)} \sum_{y(p)} \legendre{xy}{p}^c \legendre{x^{n-h} + t^{k+(n-h)}}{p} \legendre{y^{n-h} + t^{k+(n-h)}}{p}\\
&=
p\cdot A_{\X_{n,h,k+(n-h)},2}(p).
\end{split}
\end{align}

Next, we prove a lemma that greatly simplifies the calculation.
\begin{lemma}
	Suppose \(\gcd(k,n-h,p-1) = 1\).
	Then, \(A_{\X_{n,h,k},2}(p) = A_{\X_{n,h,1},2}(p)\).
\end{lemma}
\begin{proof}
	By Lemma \ref{gcd-lemma}, there exists an \(m\) such that \(\gcd(m,p-1) = 1\) and \(m \equiv k \bmod (n-h)\).
	Thus
	\begin{align}
	\begin{split}
	p\cdot A_{\X_{n,h,k},2}(p) 
	&=
	p\cdot A_{\X_{n,h,m},2}(p)\\ 
	&=
	\sum_{t = 1}^{p-1} \sum_{x(p)} \sum_{y(p)} \legendre{xy}{p}^c \legendre{x^{n-h} + t^m}{p} \legendre{y^{n-h} + t^m}{p}\\ 
	&=
	\sum_{t = 1}^{p-1} \sum_{x(p)} \sum_{y(p)} \legendre{xy}{p}^c \legendre{x^{n-h} + t}{p} \legendre{y^{n-h} + t}{p}\\ 
	&=
	p\cdot A_{\X_{n,h,1},2}(p).
	\end{split}
	\end{align}
\end{proof}
Finally, it remains to compute \(A_{\X_{n,h,1},2}(p)\). We now turn our attention to this.
\begin{theorem}
	Suppose \(\gcd(k,n-h,p-1) = 1\).
	Then
	\begin{align}
	p\cdot A_{\X_{n,h,1},2}(p)
	&=
	\begin{cases}
	(\gcd(p-1,n-h) - 1)(p^2-p) & \text{if } h ~ \textnormal{even}\\
	\gcd(n-h,p-1) (p^2-p)      & \text{if } h ~ \textnormal{odd and} ~ \nu_2(p-1) > \nu_2(n-h)\\
	0                          & \textnormal{otherwise}.
	\end{cases}
	\end{align}
\end{theorem}
\begin{proof} 
We have
\begin{align}
\begin{split}
p\cdot A_{\X_{n,h,1},2}(p)
&=
\sum_{t = 1}^{p-1} \sum_{x(p)} \sum_{y(p)} \legendre{xy}{p}^c \legendre{x^{n-h} + t}{p} \legendre{y^{n-h} + t}{p}\\ 
&=
\sum_{x,y(p)} \legendre{xy}{p}^c \sum_{t(p)} \legendre{x^{n-h} + t}{p} \legendre{y^{n-h} + t}{p}\\ 
&=
\sum_{\substack{x,y (p) \\ x^{n-h} \equiv y^{n-h}\bmod p}} \legendre{xy}{p}^c (p-1)
- \sum_{\substack{x,y (p) \\ x^{n-h} \not \equiv y^{n-h}\bmod p}} \legendre{xy}{p}^c\\
&=
p \sum_{\substack{x,y (p) \\ x^{n-h} \equiv y^{n-h}\bmod p}} \legendre{xy}{p}^c
- \sum_{x,y (p)} \legendre{xy}{p}^c\\
\end{split}
\end{align}

First, suppose that \(c = 2\).
Then, by Lemma \ref{power}, we have
\begin{align}
\begin{split}
p\cdot A_{\X_{n,h,1},2}(p)
&=
p N_p(x^{n-h} \equiv y^{n-h}) - p^2\\
&=
p (\gcd(p-1,n-h)(p-1)+1)) - p^2\\
&=
(\gcd(p-1,n-h) - 1)(p^2-p).
\end{split}
\end{align}

Second, suppose that \(c = 1\).
Then Lemma \ref{double_sums} yields
\begin{align}
\begin{split}
p\cdot A_{\X_{n,h,1},2}(p)
&=
p \sum_{\substack{x,y (p) \\ x^{n-h} \equiv y^{n-h}\bmod p}} \legendre{xy}{p}\\
&=
\begin{cases}
\gcd(n-h,p-1) (p^2-p) & \nu_2(p-1) > \nu_2(n-h)\\
0                   & \textnormal{otherwise},
\end{cases}
\end{split}
\end{align}
as desired.
\end{proof}

\end{document}